\documentclass[12pt]{amsart}

\usepackage{amsmath,amssymb,amsthm,graphicx,color,amscd}
\usepackage[usenames,dvipsnames]{xcolor}
\usepackage{enumitem}

\usepackage{hyperref}

\numberwithin{equation}{section}
\hypersetup{bookmarksdepth=3}

\theoremstyle{plain}
\newtheorem{theorem}{Theorem}[section]
\newtheorem{proposition}{Proposition}[section]
\newtheorem{lemma}{Lemma}[section]

\theoremstyle{definition}
\newtheorem{definition}{Definition}[section]

\theoremstyle{remark}
\newtheorem{remark}{Remark}[section]

\newcommand{\lloc}{L^1_{loc}}
\newcommand{\lpl}{L^p_{loc}}
\newcommand{\dimm}{\operatorname{dim}}
\newcommand{\dett}{\operatorname{det}}

\title{On the smoothness of $C^1$-contact maps in $C^\infty$-rigid Carnot groups}

\author{Jona Lelmi}
\address[Jona Lelmi]{Institut für angewandte Mathematik, Universität Bonn, Endenicher Allee 60, 53115 Bonn, Germany}
\email{jona.lelmi@uni-bonn.de}

\begin{document}

\maketitle

\begin{abstract}
We show that in any $C^\infty$-rigid Carnot group in the sense of Ottazzi - Warhurst, $C^1$-contact maps are automatically smooth.
\end{abstract}

\tableofcontents

\section{Introduction and statement of results}

In \cite{Ottazzi2011} Ottazzi and Warhurst introduced a notion of rigidity of Carnot groups. They say that a Carnot group $G$ is $C^k$-rigid provided for any connected open subset $U \subset G$ the space $\text{Con}^k(U)$ of $C^k$-contact maps defined on $U$ is finite dimensional. They also show that $C^\infty$-rigid groups are $C^2$-rigid. It is then natural to raise the question of whether this is true also in the $C^1$-case or even for rigidity defined using weaker notions of contact maps. This is also remarked in the recent work of Kleiner, M\"uller and Xie \cite{Kleiner2020}. The fact that $C^{\infty}$-rigid groups are $C^2$-rigid relies on the fact that the pushforward of a contact field by a $C^2$ contact map $f$ is still a contact field. Using a local frame of contact fields one can then pushforward by $f$ this to a new frame in the image. In $C^{\infty}$-rigid Carnot groups one can easily see, by smoothing, that contact fields are smooth. One can then combine the observations made above to construct local charts in which the map $f$ becomes the identity, thus a smooth map. For the case $k = 1$ there is a lack of regularity to use the same argument: indeed the key technical point in the case $k \ge 2$ is the fact that if $f:U \to V$ is a $C^2$-contact diffeomorphism, then $[f_{*}Z, X] = f_{*}[Z, f^{-1}_*X]$ whenever $Z$ is a contact field and $X$ is horizontal. This is used to show that the pushforward by a $C^2$-contact diffeomoprhism maps contact fields to contact fields. In this paper, we extend the result to the $C^1$-case. Namely, we prove the following theorem.
\begin{theorem}\label{c1rigTheo}
Let $G$ be a $C^{\infty}$-rigid Carnot group, then every $C^1$-contact map is smooth.
\end{theorem}
The strategy for the proof is somehow similar to the one mentioned above, the heart of the matter is a new definition of weak contact field (see Definition \ref{weakCofield}): this is essentialy obtained by reading the contact field equation in the weak sense. We will next show that in any Carnot group $G$, the pushforward by a $C^1$-contact map sends contact fields to weak contact fields, this is the content of Theorem \ref{pushfoc1}. Our definition of weak contact field is obtained making some observations on the duality action of contact fields on differential forms. This gives a definition of weak contact fields which involves a left invariant integral equation, thus it allows for smoothing. In rigid groups this fact, together with a priori estimates coming from the finite dimensionality of the space of contact fields, allows us to conclude that weak contact fields are actually smooth contact fields, this is the content of Theorem \ref{regularityWeakcof}. Combining these facts one can use the same argument as in the case $k \ge 2$ to show that $C^1$-contact maps are smooth.
\\
The strategy of the proof is flexible and it applies also to maps which are contact in a weaker sense. Unfortunately, it seems that the strategy does not apply without adding any assumption on the \textit{full} gradient: so one cannot treat the case of weak contact maps or quasiconformal maps. Extensions and limits of the strategy will be addressed in the Master thesis of the author \cite{JLMaster}.
\\
After this work was completed the author learned that in a very recent work A. Austin \cite{Austin2020} has proved Theorem \ref{c1rigTheo} for the special case of $(2,3,5)$ distributions. His proof already contains the idea of defining a certain notion of generalized contact fields that are shown to be smooth in every $C^{\infty}$-rigid Carnot group. Nevertheless, we find that our definition captures better the properties of such vector fields and the coordinate free definition allows us to conclude smoothness of those fields in a transparent way. To conclude the proof of Theorem \ref{c1rigTheo} in the special case of the $(2,3,5)$ distributions the author of \cite{Austin2020} clearly uses the special algebraic structure of them, while here we prove the statement for general $C^{\infty}$-rigid Carnot groups using mainly the properties of the dual action of contact fields on differential forms.

\subsection*{Structure of the paper}
The paper is organized as follows: in Section \ref{Prelim} we recall some basic results and definitions about Carnot groups, this will be useful to set up the notation we use later on. In Section \ref{wConF} we define the notion of weak contact field and motivate it by showing that every contact field is a weak contact field. We then prove that in $C^{\infty}$-rigid Carnot groups, weak contact fields are smooth. This will require some properties of the smoothing operations applied to vector fields and differential forms: these operations are introduced in Section \ref{smoothing}, where we also collect some of their basic properties. In Section \ref{c1rig} we finally prove the main result of the paper: first we show that the pushforward by a $C^1$-contact map sends contact fields to weak contact fields and then we deduce smoothness of those maps in the setting of $C^{\infty}$-rigid groups.

\subsection*{Acknowledgements} I wish to thank Bruce Kleiner and Stefan M\"uller for interesting discussions and for helpful suggestions that improved the original exposition of the argument.
\section{Preliminaries}\label{Prelim}
Hereafter we introduce the basic objects of our invastigation, this will serve also to set the notation we will be using later on. The notation mostly follows the paper of Ottazzi and Warhurst \cite{Ottazzi2011}.

\subsection*{Carnot groups}
Hereafter $G$ will always denote a Carnot group. A Carnot group is a connected, simply connected Lie group whose Lie algebra $\mathfrak{g}$ is stratified in the following sense: there exist subspaces $\mathfrak{g}_{-1}, . . . , \mathfrak{g}_{-s}$ such that
\begin{enumerate}
\item $\mathfrak{g} = \mathfrak{g}_{-1} \oplus . . . \oplus \mathfrak{g}_{-s}$.
\item For any $i = 1, . . . , s$ we have $[\mathfrak{g}_{-1}, \mathfrak{g}_{-i}] = \mathfrak{g}_{-i-1}$, where we set $\mathfrak{g}_{-s-1} = \{0\}$.
\end{enumerate}
We will denote with $n := \dimm(\mathfrak{g})$ the topological dimension of the group, $d_l := \dimm(\mathfrak{g}_{-l})$ will denote the dimension of the layer $l$ and $\nu = \sum_{l=1}^s ld_l$ will denote the homogeneous dimension. We will usually denote by $\{ X_{-l,v},\ l= 1, . . . , s,\ v = 1, . . . , d_l\}$ a basis of $\mathfrak{g}$ adapted to the stratification. If confusion does not arise, we will not distinguish between an element of $\mathfrak{g}$ and the left-invariant vector field associated to it. We usually denote by $\{ \sigma_{-l,v},\ l = 1, . . . , s,\ v = 1, . . . , d_l\}$ the basis of one forms dual to $\{X_{-l,v}\}$. As for vector fields, we do not distinguish between elements of $\bigwedge^{\bullet} \mathfrak{g}$ and the left-invariant forms generated by them. We will denote by $\delta_{t}$ the usual dilations operators: given $t > 0$, $\delta_t$ is the unique Lie algebra automorphism such that $\delta_t(X_{-l}) = t^l X_{-l}$ for any $X_{-l} \in \mathfrak{g}_{-l}$ and any $l = 1, . . . , s$.
\\
Given any point $x \in G$ we will denote by $\ell_x$ the left multiplication by $x$ and with $r_x$ the right multiplication by $x$.
The horizontal bundle of the group is the bundle whose fibers are obtained by left translating $\mathfrak{g}_{-1}$. The bundle map is given by the projection on the base point. We will denote this bundle by $\mathcal{H}$. After introducing an inner product on $\mathfrak{g}_{-1}$ so that the basis $\{X_{-1,i},\ i= 1, . . . , d_1\}$ becomes orthonormal, we can define the Carnot-Carathéodory distance in the usual way. We denote this distance by $d_{cc}$.
\\
Given a Carnot group $G$ and an open subset $U \subset G$, we will denote by $\Gamma(U)$ the space of measurable sections of $TG$ defined on $U$ and with $\Omega^{k}(U)$ (resp. $\Omega^{\bullet}(U)$) the space of measurable sections of $\bigwedge^{k}G$ (resp. $\bigwedge^{\bullet}G$) defined on $U$. Given $k \ge 1$, $X \in \Gamma(U)$ and $\omega \in \Omega^k(U)$ we define $i_X\omega \in \Omega^{k-1}(U)$ by
\begin{equation}
(i_X\omega)_p (v_1, . . . , v_{k-1}) = \omega_p(X_p, v_1, . . . , v_{k-1})
\end{equation}
for any $v_1, . . . , v_{k-1} \in T_pG$ and a.e. $p \in U$.
\begin{definition}\label{verFor}
Let $U \subset G$ an open subset of a Carnot group. A measurable $1$-form $\eta \in \Omega^1(U)$ is called vertical if 
\begin{equation}
\eta_x(X_x) = 0
\end{equation}
for a.e. $x \in U$ and every $X_x \in \mathcal{H}_x$.
\end{definition}
The primary objects of investigation are contact maps.
\begin{definition}
Let $k \in \mathbf{N} \cup \{\infty\}$. Given $U \subset G$ open, a $C^k$-local diffeomorphism $f: U \to G$ is said to be a $C^k$ contact map if the differential preserves the horizontal bundle, i.e. $d_xf(\mathcal{H}_x) \subset \mathcal{H}_{f(x)}$ for any $x \in U$.
\end{definition}
Following Ottazzi and Warhurst we define rigid Carnot groups as follows.
\begin{definition}
A Carnot group $G$ is called $C^k$-rigid if given any connected open subset $U \subset G$ the space
\begin{equation}
\text{Con}^k(U) := \left\{f:U \to G:\ f\ C^k\text{-contact map} \right\}
\end{equation}
is finite dimensional.
\end{definition}
On the infinitesimal level, we have the following definition.
\begin{definition}\label{coF}
Let $Z$ be a locally defined $C^1$-vector field, let $\phi^Z_t$ its locally defined flow. Then $Z$ is called a contact field if $\phi^Z_t$ is a $C^2$-contact map for $t$ small enough.
\end{definition}
\begin{remark}\label{cofFinite}
Ottazzi and Warhurst proved in \cite{Ottazzi2011} that a Carnot group is $C^2$-rigid if and only if the space of germs of contact fields at a point $p$ is finite dimensional for every choice of $p \in G$.
\end{remark}
\begin{remark}\label{contactFieldsFrame}
Let us point out the following facts.
\begin{enumerate}
\item It is not hard to check that a locally defined $C^1$ vector field $Z$ is a contact field if and only if $[Z, X_{-1}] \in \mathcal{H}$ for any $X_{-1} \in \mathcal{H}$.
\item We can always find a global frame of contact fields. Indeed define $X^{R}_{-l,v}(x) = d_er_xX_{-l,v}$, then $X^{R}_{-l,v}$ is a global frame of right invariant vector fields. Since right invariant fields commute with left invariant fields we have
\begin{equation}
[X_{-l,v}^R, X_{-1,i}] = 0 \in \mathcal{H}.
\end{equation}
Thus $\{X_{-l,v}^R\}$ is a family of contact fields.
\end{enumerate}
\end{remark}

\subsection*{Weights.}
We next recall the definition of weights of covectors.
\begin{definition}
Let $\mathfrak{g}$ be the Lie algebra of a Carnot group $G$. A nonzero covector $\omega$ in $\bigwedge^k \mathfrak{g}$ is said to have weight $w$ if, for every $t > 0$, $(\delta_t)_* \omega = t^w \omega$. We write $wt(\omega) = w$.
\end{definition}
\begin{remark}
It is easy to see that $\omega$ in $\bigwedge^1\mathfrak{g}$ has weight $w$ if and only if
\begin{equation}
\omega = \sum_{\alpha = 1}^{d_w} c_{\alpha}\sigma_{w, \alpha}
\end{equation}
for some constants $c_{\alpha} \in \mathbf{R}$. Indeed, if $v \in \mathfrak{g}_{-l}$ one has
\begin{equation}
\begin{split}
\omega(\frac{1}{t^l}v) &= \omega(\delta_{\frac{1}{t}}(v)) 
\\ &= (\delta_t)_{*}\omega(v) 
\\ &= t^{w}\omega(v).
\end{split}
\end{equation} 
Letting $v$ vary on $\mathfrak{g}_{-l}$ this forces $\omega(v) = 0$ if $-l \neq w$. A similar reasoning shows that $\omega \in \bigwedge^k \mathfrak{g}$ is of weight $w$ if and only if
\begin{equation}
\omega = \sum_{I} c_I \theta_I
\end{equation}
where $c_I \in \mathbf{R}$ and $\theta_I$ is a wedge product of covectors such that their weights add up to $w$. In particular, if $\omega \in \bigwedge^k \mathfrak{g}$ is non zero we have $-\nu \le wt(\omega) \le -k$.
\end{remark}
For measurable forms, we give the following definition.
\begin{definition}
Let $U \subset G$ an open subset of a Carnot group. Let $\omega \in \Omega^{\bullet}(U)$ a measurable section of $\bigwedge\nolimits^{\bullet} G$, if $\omega$ is not a.e. vanishing we define
\begin{equation}
wt(\omega) := \operatorname{esssup}_{x \in U} wt(\omega(x)).
\end{equation}
\end{definition}
\begin{remark}
It is clear that, with $\omega$ as above, $\omega(x)$ may be identified with an element of $\bigwedge^{\bullet} \mathfrak{g}$, so the definition is well posed.
\end{remark}

\subsection*{Structure constants and the contact field equation}. We next set up the notation for the structure constants of the group. Let $\mathfrak{g} = \mathfrak{g}_{-1} \oplus . . .  \oplus \mathfrak{g}_{-s}$ be the Lie algebra of a Carnot group $G$. Let, as above, $\{X_{-l,v}\}$ be a basis adapted to the stratification. Then for any $l = 1, . . . , s-1$, $v = 1, . . . d_l$ and $j = 1, . . . , d_1$ there are constants $\alpha_{v,j}^{l,1,k}$ (hereafter called \textit{structure constants}) such that
\begin{equation}
[X_{-l,v}, X_{-1,j}] = \sum_{k=1}^{d_{l+1}} \alpha_{v,j}^{l,1,k}X_{-l-1,k}.
\end{equation}
Let $Z$ be a locally defined $C^1$-vector field and define its coefficients $z_{-l,v}$ by
\begin{equation}
Z = \sum_{l=1}^s \sum_{v=1}^{d_l} z_{-l,v} X_{-l,v}.
\end{equation}
Then $z_{-l,v}$ are $C^1$ and it is not hard to see that the contact condition of Remark \ref{contactFieldsFrame} for $Z$ is equivalent to
\begin{equation}\label{coCof}
X_{-1,j}(z_{-l,k}) = \sum_{r=1}^{d_{l-1}} z_{-l+1, r} \alpha_{r, j}^{l-1,1,k}\ \forall 1 \le j \le d_1,\ 2 \le l \le s,\ 1 \le k \le d_l.
\end{equation}
We will refer to (\ref{coCof}) as the contact field equation written in coordinates.

\section{Smoothing}\label{smoothing}

In this section we collect some results about smoothing on Carnot groups. First we recall the smoothing of functions defined on Carnot groups by means of convolution with a mollifying kernel. This amounts to perform an average over translations of the function. Motivated by this, we can define smoothing operations on forms and vector fields so that these behave well in connection to the classical operations of Cartan calculus. 

\subsection{Smoothing functions}

We recall here the usual smoothing of functions on Carnot groups. The material here is standard and can be found in \cite{Folland1982, Grafakos2014}. Let $G$ be a Carnot group, let $\rho \in C^{\infty}_c(G)$ such that

\begin{equation}
\int_G \rho dx = 1,\ \ 0 \le \rho \le 1,\ \ \rho(y) = \rho(y^{-1}),\ \ supp(\rho) \subset B_{cc}(e, 1).
\end{equation}
The existence of such a function is easily verified. Here $B_{cc}(e,1)$ is the unit ball with respect to the Carnot-Carathéodory distance centered at the identity. Define, for $\epsilon > 0$, the function
\begin{equation}
\rho_{\epsilon}(x) = \epsilon^{-\nu} \rho\left(\delta_{\frac{1}{\epsilon}}(x)\right).
\end{equation}
Let $dx$ denote the biinvariant Haar measure of $G$. For a function $f \in L^{1}_{\text{loc}}(G)$ we set
\begin{equation}
\rho_\epsilon * f(x) := \int_{G} \rho_\epsilon (x y^{-1})f(y) dy = \int_G\rho_{\epsilon}(y^{-1})f(yx)dx = \int_{G}\rho_{\epsilon}(y)f(y^{-1}x) dy
\end{equation}
observe that for the second and third equality we used the bi-invariance of measure and the invariance under inversion. If $U \subset G$ is open, and $f \in L^{1}_{\text{loc}}(U)$, then we can define $\rho_{\epsilon} * f(x) = \rho_{\epsilon} * \tilde{f}(x)$, where $\tilde{f}$ is the extension by zero of $f$. Then the following properties hold.
\begin{proposition}\label{prop conv}
For $f \in L^1_{\text{loc}}(U)$ we have
\begin{enumerate}
\item $\rho_{\epsilon} * f \in C^{\infty}(G)$.
\item If $1 \le p < \infty$ and $f \in L^p(U)$ then $\rho_{\epsilon} * f \to f$ in $L^p(U)$.
\item $\int_{G} (\rho_{\epsilon} * f) g dx = \int_{G} (\rho_{\epsilon} * g) f dx$ for $g \in L^{\infty}(G)$ with compact support, or for $g \in L^{\infty}(G)$ if in addition $f \in L^1(U)$.
\item If $f \in C^{0}(U)$ then $\rho_{\epsilon} * f \to f$ locally uniformly in $U$.
\end{enumerate}
\end{proposition}
The proofs of the statements are completely analogous to the ones on $\mathbf{R}^n$, see \cite{Folland1982}.
\subsection{Smoothing forms and vector fields}
In this subsection $U \subset G$ will be open.
\begin{definition}\label{sforms}
Let $\theta \in \Omega^{\bullet}(U)$ be a form with $\lloc(U)$ coefficients. For $\epsilon > 0$ and $x \in U$ define
\begin{equation}
\theta^{\epsilon}_{x} := \int_G (\ell_y^{*}\theta)_x\rho_{\epsilon}(y^{-1})dy
\end{equation}
\end{definition}
\begin{definition}\label{svfields}
Let $X \in \Gamma(U)$ be a vector field with $\lloc(U)$ coefficients. For $\epsilon > 0$ and $x \in U$ define
\begin{equation}
X^{\epsilon}_x := \int_{G} ((\ell_y)_*X)_x\rho_{\epsilon}(y) dy
\end{equation}
\end{definition}
\begin{remark}
It should be noticed that in the setting of Definition \ref{sforms} we have $(\ell_y^{*}\theta)_x \in \bigwedge^{\bullet}(T_xG^{*})$, thus the integral is well defined as an element in this vector space. Similar reasoning applies to Definition \ref{svfields}.
\end{remark}
\begin{proposition}[Properties of smoothing]\label{prop sforms} The following assertions hold.
\begin{enumerate}
\item If $\theta$ (respectively $X$) is as in Definition \ref{sforms} (resp. Definition \ref{svfields}) then $\theta^{\epsilon}$ (resp. $X^{\epsilon}$) is smooth.
\item If $\theta$ is a $k$-form with $\lloc(U)$ coefficients and $\beta$ is a compactly supported $n-k$ form with $\text{L}^{\infty}(U)$ coefficients
\begin{equation}
\int_U \theta^{\epsilon} \wedge \beta = \int_U \theta \wedge \beta^{\epsilon}.
\end{equation}
\item If $\alpha$ is a left invariant $(k+1)$-form, $\beta \in \Omega^{n-k}(U)$ has compact support and $X$ is as in Definition \ref{svfields}
\begin{equation}
\int_{U} i_{X^{\epsilon}}(\alpha) \wedge \beta = \int_{U} i_X(\alpha) \wedge \beta^{\epsilon}
\end{equation}
for every $\epsilon > 0$ small enough.
\item If $\theta \in \Omega^{k}(U)$ with coefficients in $\lloc(U)$ has a distributional exterior differential $d\theta \in \Omega^{k+1}(U)$ with $\lloc(U)$ coefficients, then $d\theta^{\epsilon} = (d\theta)^{\epsilon}$.
\item If $\theta \in \Omega^{\bullet}(U)$ (resp. $X \in \Gamma(U)$) has coefficients in $\lpl(U)$, then $\theta^{\epsilon} \to \theta$ (resp. $X^{\epsilon} \to X$) in $\lpl(U)$.
\item If $\theta \in \Omega^{\bullet}(U)$ (resp. $X \in \Gamma(U)$) is continuous the convergence in part (5) holds uniformly on compact subsets of $U$.
\end{enumerate}
\end{proposition}
\begin{proof}
Parts (1), (5) and (6) follow easily by writing the smoothing operation in coordinates. First observe that since the smoothing operation is linear in $\theta$ we may reduce to the case of forms with fixed degree $k$, thus we assume $\theta \in \Omega^{k}(U)$. Write
\begin{equation}
\theta = \sum_{I} \theta_{I} \sigma_{I}
\end{equation}
where $\sigma_{I}$ are left invariant forms of degree $k$ and $\theta_{I} \in \lloc(U)$. Then
\begin{equation}
\begin{split}
\theta_{x}^{\epsilon} &= \sum_{I} \int_{G} \theta_{I}(\ell_{y}(x))\sigma_{I,x}\rho_{\epsilon}(y^{-1})dy
\\ & = \sum_{I} \rho_{\epsilon} \ast \theta_I (x)\sigma_{I,x}
\end{split}
\end{equation}
thus the claims follow by an application of Proposition \ref{prop conv}. The argument for the smoothing of vector fields is analogous.
\\
For (3) we let $X_1, . . . , X_n$ be a basis $\mathfrak{g}$. If $\sigma_1, . . . , \sigma_n$ is the dual basis of $\{X_{i}\}$, by the uniqueness of the Haar measure we can assume (up to a constant)
$$
\bigwedge_{i=1}^n \sigma_{i} = dx.
$$
Thus
\begin{equation}\label{calcPart3}
\begin{split}
&\int_{U} i_{X^{\epsilon}}(\alpha) \wedge \beta = \int_{U} \left(i_{X^{\epsilon}}(\alpha) \wedge \beta\right) (X_1, . . . , X_n)dx 
\\& = \sum_{\gamma \in \mathfrak{S}n} \text{sgn}(\gamma)\int_{U}i_{X^{\epsilon}}(\alpha)(X_{\gamma(1),x}, . . . , X_{\gamma(k), x})\beta(X_{\gamma(k+1),x}, . . . , X_{\gamma(n),x})dx
\\ &= \sum_{\gamma \in \mathfrak{S}n} \text{sgn}(\gamma)\int_{U}\alpha_x\left(\int_G((\ell_y)_{*}X)_x\rho_{\epsilon}(y)dy, X_{\gamma(1),x}, . , X_{\gamma(k), x}\right)\beta_x(X_{\gamma(k+1),x}, . , X_{\gamma(n),x})dx
\\ &= \sum_{\gamma \in \mathfrak{S}n} \text{sgn}(\gamma)\int_G\int_{U}\alpha_x((\ell_y)_{*}X)_x, X_{\gamma(1),x}, ... , X_{\gamma(k), x})\beta_x(X_{\gamma(k+1),x}, ... , X_{\gamma(n),x})dx\rho_{\epsilon}(y)dy.
\end{split}
\end{equation}
Now we can perform the change of variables $x \to yx$ in the inner integral, observe that by left invariance $X_{i,yx} = (d\ell_y)_x X_{i,x}$ and $(\ell_{y}^*\alpha)_x = \alpha_x$. Recall also that $\rho_{\epsilon}(y) = \rho_{\epsilon}(y^{-1})$. Putting things together we get that the last line of (\ref{calcPart3}) is equal to
\begin{equation}
\begin{split}
& \sum_{\gamma \in \mathfrak{S}_n}\text{sgn}(\gamma) \int_{G} \int_{U} i_{X}(\alpha)_x( X_{\gamma(1),x}, ... , X_{\gamma(k), x})(\ell_{y}^*\beta)_x(X_{\gamma(k+1),x}, ... , X_{\gamma(n),x})\rho_{\epsilon}(y)dxdy
\\ &= \int_{U} i_{X}(\alpha) \wedge \beta^{\epsilon}.
\end{split}
\end{equation}
The argument for part (2) is just an application of Fubini's theorem and the invariance of the Haar measure under inversion.\\
Finally part (4) follows by the fact that $d$ commutes with $\ell_{y}^*$ in the case $\theta$ is smooth. For the general case pick any smooth compactly supported $(n-k-1)$-form $\alpha$, then using part (2) and the result for smooth forms we get
\begin{equation}
\begin{split}
\int_U \theta^{\epsilon} \wedge d\alpha &= \int_U \theta \wedge (d\alpha)^{\epsilon}
\\ & = \int_U \theta \wedge d\alpha^{\epsilon}
\\ & = (-1)^k\int_U d\theta \wedge \alpha^{\epsilon}
\\ & = (-1)^k\int_U (d\theta)^{\epsilon} \wedge \alpha.
\end{split}
\end{equation}
\end{proof}
\section{Weak contact fields}\label{wConF}
In this section we define the notion of weak contact fields and prove that, on every rigid Carnot group, they are smooth.

\subsection{Weak contact fields}
Let $Z$ be a $C^1$-contact field on an open set $U$. Let $\phi_t$ be the local flow of $Z$. The maps $\phi_t$ are contact maps. Thus if $\eta$ is a smooth vertical $1$-form in the sense of Definition \ref{verFor}, the Lie derivative 
\begin{equation}
\mathcal{L}_{Z}\eta(x) = \lim_{t \to 0} \frac{(\phi_t^*\eta)_x - \eta_x}{t}
\end{equation} 
is still a vertical $1$-form. This amounts to saying that $\mathcal{L}_Z\eta(X) = 0$ for any horizontal vector field $X$. Using Cartan's magic formula one can rewrite this as
\begin{equation}
\left( i_Z d\eta + d i_Z\eta \right)(X) = 0.
\end{equation}
Let now $\sigma$ be the volume form on $G$. Then for degree reasons, using also the Leibniz rule for the interior multiplication $i_X$, we have 
\begin{equation}\label{cartancodeg1}
\begin{split}
\left( i_Z d\eta + d i_Z\eta \right)\wedge i_X\sigma = &-i_X\left( (i_Z d\eta + d i_Z\eta) \wedge \sigma \right)
\\ &+ i_X\left( i_Z d\eta + d i_Z\eta \right) \wedge \sigma = 0. 
\end{split}
\end{equation}
Now let $\beta$ be a smooth codegree $1$ form of weight $-\nu + 1$ with compact support, in coordinates $\beta$ may be written as
\begin{equation}
\beta = \sum_{i=1}^{d_1} \beta_i \hat{\sigma}_{-1,i},\ \beta_i \in C_c^{\infty}(U)
\end{equation}
where we set
\begin{equation}
\hat{\sigma}_{-1,i} := \bigwedge_{(l,v) \neq (1,i)} \sigma_{-l,v}.
\end{equation}
Then $i_X\sigma = \beta$ if
\begin{equation}
X = \sum_{i=1}^{d_1} (-1)^{i-1}\beta_iX_{-1,i}.
\end{equation}
Inserting into this identity into (\ref{cartancodeg1}) and integrating on $U$ we get
\begin{equation}
0 = \int_U (i_Z d\eta + d i_Z\eta) \wedge \beta = \int_U i_Z (d\eta) \wedge \beta - \int_U i_Z\eta \wedge d\beta.
\end{equation}
It is then natural to give the following definition.
\begin{definition}\label{weakCofield}
Let $U \subset G$ be open. A vector field $Z \in \Gamma(U)$ with coefficients in $\lloc(U)$ is called a weak contact field provided that
\begin{equation}\label{eqweakcof}
\int_{U} i_{Z}(d\eta) \wedge \beta - \int_{U} i_Z(\eta) \wedge d\beta = 0
\end{equation}
for any smooth vertical $1$-form $\eta \in \Omega^1(U)$ and any smooth compactly supported form $\beta \in \Omega^{n-1}(U)$ of weight $-\nu + 1$.
\end{definition}
\begin{remark}\label{remarkWeakcof}
A few comments are in order.
\begin{enumerate}
\item Observe now that in order to check that a locally integrable vector field $Z$ is a weak contact field it is sufficient to show that (\ref{eqweakcof}) holds for every form $\beta$ as above and every $\eta$ which is vertical and left invariant. To see this, observe that if $\tilde{\eta}$ is a smooth vertical $1$-form on $U$, then we can write $\tilde{\eta}$ as
\begin{equation}
\tilde{\eta} = \sum_{I} f_I \eta_I
\end{equation}
where $f_I \in C^{\infty}(U)$ and $\eta_I$ are left-invariant, vertical $1$-forms. By linearity we can thus verify the claim for $\tilde{\eta} = g\eta$, where $g \in C^{\infty}(U)$ and $\eta$ is a left-invariant, vertical $1$-form. Observe that $d\tilde{\eta} = dg \wedge \eta + gd\eta$, thus $i_Zd\tilde{\eta} = i_Zdg \wedge \eta - dg \wedge i_Z\eta + gi_Zd\eta$. Using the result for the left-invariant form $\eta$ and for $\tilde{\beta} = g\beta$, and the fact that $\eta \wedge \beta = 0$ we get
\begin{equation}
\begin{split}
& \int_U i_Zd\tilde{\eta} \wedge \beta =
\\ & = \int_U i_Z(dg) \wedge \eta \wedge \beta - \int_U dg \wedge i_Z\eta \wedge \beta + \int_U i_Z\eta \wedge d(g\beta)
\\ & = - \int_U dg \wedge i_Z\eta \wedge \beta + \int_U i_Z\tilde{\eta} \wedge d\beta + \int_U dg \wedge i_Z\eta \wedge \beta
\\ & = \int_U i_Z\tilde{\eta} \wedge d\beta.
\end{split}
\end{equation}
\item One can check that Definition \ref{weakCofield} is equivalent to requiring that the contact field equation (\ref{coCof}) holds in the weak sense for $Z$, i.e. writing $Z = \sum_{l=1}^s \sum_{v =1}^{d_l}z_{-l,v}X_{-l,v}$ we have
\begin{equation}
\int_{U}z_{-l,k} X_{-1,j}\phi dx = -\sum_{r=1}^{d_{l-1}}\int_{U}z_{-l+1,r}\alpha_{r,j}^{l-1,1,k}\phi dx
\end{equation}
for any $\phi \in C^{\infty}_c(U)$, for every $l \ge 2$, $1 \le k \le d_l$ and any $1 \le j \le d_1$. Indeed, in (\ref{eqweakcof}) , one can choose $\eta = \sigma_{-l,k}$ and $\beta = \phi \hat{\sigma}_{-1,j}$. In particular every $C^1$ contact field is a weak contact field and, conversely, every weak contact field of class $C^1$ is a contact field. 
\end{enumerate}
\end{remark}

\subsection{Regularity of weak contact fields}
Here we prove the following regularity result.
\begin{theorem}\label{regularityWeakcof}
Let $G$ be a $C^{\infty}$-rigid Carnot group. Let $Z \in \Gamma(U)$ be a weak contact field, then $Z$ can be redefined on a set of measure zero so that it becomes smooth. It follows that $Z$ is a contact field in the sense of Definition \ref{coF}. If $Z$ is already continuous, then $Z$ is smooth.
\end{theorem}
\begin{proof}
Let $Z^{\epsilon}$ be the smoothed version of $Z$ as introduced in Definition \ref{svfields}. We will show below that for any $\epsilon < \tilde{\epsilon}$, $Z^{\epsilon}$ is a weak contact field on $U_{\tilde{\epsilon}} := \{ x \in U : d_{cc}(x, U^c) > 2\tilde{\epsilon} \}$. If we assume this to be true, the smoothness of $Z^{\epsilon}$ together with Remark \ref{remarkWeakcof} imply that $Z^{\epsilon}$ are contact fields on $U_{\tilde{\epsilon}}$. As pointed out in Remark \ref{cofFinite} the space of contact fields on $U_{\tilde{\epsilon}}$ is finite dimensional by the rigidity of the group, it follows that any norm on this space is comparable. Since $Z^{\epsilon} \to Z$ on $L^1(U_{\tilde{\epsilon}})$ it follows that $Z^{\epsilon}$ is a Cauchy sequence in $C^{j}(U_{\tilde{\epsilon}})$ for any $j \in \mathbf{N}$. Since a subsequence $Z^{\epsilon_j}$ converges to $Z$ a.e. it follows that we can redefine $Z$ on a set of measure zero on $U_{\tilde{\epsilon}}$ in a way that it becomes smooth. Applying this fact along a sequence $\tilde{\epsilon}_{k} \to 0$ we get the result. Clearly if $Z$ is already continuous we do not need to redefine it, because two continuous functions coinciding a.e. coincide everywhere.
\\
To prove that $Z^{\epsilon}$ is a weak contact field on $U_{\tilde{\epsilon}} := \{ x \in U : d_{cc}(x, U^c) > 2\tilde{\epsilon} \}$ pick $\eta$ and $\beta \in \Omega^{n-1}(U_{\tilde{\epsilon}})$ as in Definition \ref{weakCofield}. Then observing that $d\eta$ is also left invariant, we can use Proposition \ref{prop sforms} (3) and the definition of a weak contact field to get
\begin{equation}
\begin{split}
\int_{U_{\tilde{\epsilon}}} i_{Z^{\epsilon}}(d\eta)\wedge \beta &= \int_{U_{\tilde{\epsilon}}}i_Z(d\eta)\wedge \beta^{\epsilon}
\\ &= \int_{U_{\tilde{\epsilon}}} i_{Z}\eta \wedge (d\beta)^{\epsilon}
\\ &= \int_{U_{\tilde{\epsilon}}} i_{Z^{\epsilon}}\eta\wedge d\beta
\end{split}
\end{equation}
which is exactly the definition of weak contact field.
\end{proof}

\section{$C^1$-rigidity}\label{c1rig}

In this section we will prove the main result of the present paper, namely, that in any $C^{\infty }$-rigid Carnot group $C^1$-contact maps are smooth. Firstly we will show that on any Carnot group, every $C^1$-contact diffeomorphism $f: U \to V$ between two open subsets maps, via the push-forward, contact fields to weak contact fields. Then we will use Theorem  \ref{regularityWeakcof} to infer that on rigid Carnot groups $f_{*}$ maps contact fields to contact fields, and from this we will easily get the smoothness of $f$.

\subsection{The push-forward by a $C^1$-contact map in general Carnot groups}
We start by proving the following result. 
\begin{theorem}\label{pushfoc1}
Let $f: U \to V$ be a $C^1$-contact diffeomorphism between two open subsets of a given Carnot group $G$. The push-forward by $f$ maps contact fields to weak contact fields.
\end{theorem}
To prove this theorem, we use the fact that for smooth contact fields $Z$, the weak contact field equation (\ref{eqweakcof}) holds for a larger class of forms $\eta$ and $\beta$. This is the content of the following lemma.
\begin{lemma}
Let $Z$ be a contact field on $U$. Then
\begin{equation}\label{testwithcont}
\int_{U} i_{Z}(d\eta) \wedge \beta - \int_{U} i_Z(\eta) \wedge d\beta = 0
\end{equation}
for any continuous $\eta$ vertical $1$-form and any continuous compactly supported form $\beta$ of codegree $1$ and of weight $-\nu + 1$, provided $d\eta$ and $d\beta$ exist as continuous forms in the sense of distributions.
\end{lemma}
\begin{proof}
First assume that $\beta$ is smooth. Then if $\eta^{\epsilon}$ is the smoothed version of $\eta$
\begin{equation}\label{prw1}
\int_{U} i_{Z}(d\eta^{\epsilon}) \wedge \beta - \int_{U} i_Z(\eta^{\epsilon}) \wedge d\beta = 0.
\end{equation}
By Proposition \ref{prop sforms} (4) and (6) $\eta^{\epsilon} \to \eta$ and $d\eta^{\epsilon} =(d\eta)^{\epsilon} \to d\eta$ uniformly on compact subsets of $U$. One easily sees that this implies that $i_Z\eta^{\epsilon} \to i_Z\eta$ and $i_Zd\eta^{\epsilon} \to i_Zd\eta$ uniformly on compact subsets of $U$. Since $\beta$ has compact support we can pass to the limit in (\ref{prw1}) to obtain
\begin{equation}
\int_{U} i_{Z}(d\eta) \wedge \beta - \int_{U} i_Z(\eta) \wedge d\beta = 0.
\end{equation}
Now drop the smoothness assumption on $\beta$. The smoothed version $\beta^{\epsilon}$ is still compactly supported (and we can find a compact subset of $K \subset U$ such that $\text{spt}(\beta^{\epsilon})\subset K$ for all $\epsilon$ small enough) thus (\ref{testwithcont}) holds with $\beta^{\epsilon}$. Once again we may pass to the limit and this yields the claim.
\end{proof}

\begin{proof}[Proof of Theorem \ref{pushfoc1}]
Without loss of generality we may assume that $\dett Df > 0$ on $U$. Let $\tilde{\eta}$ be any left-invariant vertical $1$-form and $\tilde{\beta}$ any codegree $1$, smooth, compactly supported form of weight $-\nu + 1$ on $V$. Define forms $\eta = f^{*}\tilde{\eta}$ and $\beta = f^{*}\tilde{\beta}$. By the contact condition on $f$ the form $\eta$ is still vertical. The contact condition also implies that $\beta$ is weight $-\nu + 1$. Indeed, if $\beta$ is any form of codegree $1$, then $\beta$ has weight $-\nu+1$ if and only if $\alpha \wedge \beta = 0$ for every vertical 1-form $\alpha$. Thus if $\alpha$ is such a form we can evaluate
\begin{equation}
\alpha \wedge \beta = f^*((f^{-1})^*\alpha \wedge \tilde{\beta}) = 0
\end{equation}
where in the last equality we used that $\tilde{\beta}$ is of codegree $1$ of weight $-\nu + 1$ and we used the contact condition to infer that $(f^{-1})^*\alpha$ is still vertical.
Observe also that since $f$ is $C^1$ the pull-back commutes with $d$ in the sense of distributions, this amounts to say that $d\eta = f^*d\tilde{\eta}$ and $d\beta = f^*d\tilde{\beta}$. By the regularity assumption on $f$ these forms are continuous. We can thus insert them into (\ref{testwithcont}) to get
\begin{equation}\label{prw2}
\int_{U} i_{Z}(f^{*}d\tilde{\eta}) \wedge f^{*}\tilde{\beta} - \int_{U} i_Z(f^{*}\tilde{\eta}) \wedge f^{*}d\tilde{\beta} = 0.
\end{equation}
It is easily checked that $i_Zf^{*}d\tilde{\eta} = f^{*}i_{f_*Z}d\tilde{\eta}$, and $i_Zf^{*}\tilde{\eta} = f^{*}i_{f_*Z}\tilde{\eta}$. Thus (\ref{prw2}) reads
\begin{equation}\label{prw3}
\begin{split}
0 &= \int_{U} f^{*}\left(i_{f_*Z}d\tilde{\eta} \wedge \tilde{\beta}\right) - \int_{U} f^{*}\left(i_{f_*Z}\tilde{\eta} \wedge d\tilde{\beta}\right)
\\ & = \int_{V} i_{f_*Z}d\tilde{\eta} \wedge \tilde{\beta} - \int_{V} i_{f_*Z}\tilde{\eta} \wedge d\tilde{\beta}
\end{split}
\end{equation}
where the last equality follows by the fact that $f$ is a $C^1$-diffeomorphism with a positive Jacobian. Equation (\ref{prw3}) is the definition of weak contact field for $f_*Z$.
\end{proof}
\subsection{Smoothness of $C^1$-contact maps}
\begin{proof}[Proof of Theorem \ref{c1rigTheo}]
Let $f: U \to V$ a $C^1$-contact map. Let $p \in U$, upon restricting both $U$ and $V$ we can assume that $f$ is a $C^1$-diffeomorphism. Let $X_1, . . . , X_n$ be contact fields defined near $p$ such that $X_{1,p}, . . . , X_{n,p}$ are a basis for $T_pG$ (one may take, for example, a basis of right invariant fields, see Remark \ref{contactFieldsFrame}). Define, for a given smooth vector field $Z$ near $p$, $\phi^t_{Z}$ to be the locally defined flow of $Z$. Since by Theorem \ref{pushfoc1} and Theorem \ref{weakCofield} the fields $f_*X_i$ are smooth contact fields, we get that for $\Vert t \Vert$ small
\begin{equation}
\varphi: (t_1, . . . , t_n) \to \phi_{X_1}^{t_1} \circ . . .  \circ \phi_{X_n}^{t_n}(p)
\end{equation}
is a chart near $p$ and that
\begin{equation}
\psi: (t_1, . . . , t_n) \to \phi_{f_*(X_1)}^{t_1} \circ . . .  \circ \phi_{f_*(X_n)}^{t_n}(f(p))
\end{equation}
is a chart near $f(p)$. Since $\phi_{f_*X_i}^{t_i} \circ f = f \circ \phi_{X_i}^{t_i}$, in those coordinates $f$ is the identity, thus $f$ is smooth around $p$.
\end{proof}
\nocite{*}
\bibliography{biblio_smoothC1}{}
\bibliographystyle{plain}
\end{document}